\documentclass[reqno]{amsart}
\usepackage{CJK}
\usepackage{hyperref}
\usepackage{cite}
\usepackage{amsmath}
\usepackage{mathrsfs}

\numberwithin{equation}{section}

\newtheorem{theorem}{Theorem}[section]
\newtheorem{lemma}[theorem]{Lemma}
\newtheorem{definition}[theorem]{Definition}

\newtheorem{remark}[theorem]{Remark}

\allowdisplaybreaks

\begin{document}
	
\title[\hfil Regularity for parabolic normalized $p(x,t)$-Laplace equation] {Gradient H\"{o}lder regularity for parabolic normalized $p(x,t)$-Laplace equation}

\author[Y. Fang and C. Zhang  \hfil \hfilneg]{Yuzhou Fang and  Chao Zhang$^*$}
 
\thanks{$^*$ Corresponding author.}
 
\address{Yuzhou Fang \hfill\break School of Mathematics, Harbin Institute of Technology, Harbin 150001, China}
\email{18b912036@hit.edu.cn}

\address{Chao Zhang  \hfill\break School of Mathematics and Institute for Advanced Study in Mathematics, Harbin Institute of Technology, Harbin 150001, China}
\email{czhangmath@hit.edu.cn}


\maketitle

\begin{abstract}
We consider interior H\"{o}lder regularity of the spatial gradient of viscosity solutions to the normalized $p(x,t)$-Laplace equation
$$
u_t=\left(\delta_{ij}+(p(x,t)-2)\frac{u_i u_j}{|Du|^2}\right)u_{ij}
$$
with some suitable assumptions on $p(x,t)$, which arises naturally from a two-player zero-sum stochastic differential game with probabilities depending on space and time.
\end{abstract}

\section{Introduction}
\label{sec-1}

Let $p(x,t)\in C^1_{\rm loc}(\mathbb{R}^{n+1})$  and $1<p_-:=\inf p(x,t)\leq \sup p(x,t)=:p_+<\infty$. In this work, we investigate the higher regularity properties of viscosity solutions to the following parabolic normalized $p(x,t)$-Laplace equation
\begin{equation}
\label{1-1}
u_t(x,t)=\Delta^N_{p(x,t)}u(x,t),
\end{equation}
where $\Delta^N_{p(x,t)}$ is the normalized $p(x,t)$-Laplace operator  defined as
$$
\Delta^N_{p(x,t)}u:=\Delta u+(p(x,t)-2)\left\langle D^2u\frac{Du}{|Du|},\frac{Du}{|Du|}\right\rangle=
\left(\delta_{ij}+(p(x,t)-2)\frac{u_i u_j}{|Du|^2}\right)u_{ij}.
$$
Here the summation convention is utilized and the vector $Du$ is the gradient with respect to the spatial variable $x$.

Over the last decade, Eq. \eqref{1-1} and related normalized equations in non-divergence form have received considerable attention, partly due to the stochastic zero-sum tug-of-war games defined by Peres-Schramm-Sheffield-Wilson in  \cite{PS08,PS09} and Manfredi-Parviainen-Rossi in \cite{MPR10}.
For the case that $p(x)$ is constant,  Luiro-Parviainen-Saksman \cite{LPS13} proved the Harnack's inequality for the homogeneous normalized $p$-Laplace equation $-\Delta^N_pu=0$. Ruosteenoja \cite{Ruo16} studied the local Lipschitz continuity and Harnack's inequality for the inhomogeneous version $-\Delta^N_p u=f$, and later it was extended to the local $C^{1,\alpha}$ regularity of viscosity solutions by Attouchi-Parviainen-Ruosteenoja in \cite{APR17}.
Furthermore, Siltakoski \cite{Sil18}  considered the normalized $p(x)$-Laplace equation
\begin{equation}
\label{1-2}
\Delta^N_{p(x)}u:=\left(\delta_{ij}+(p(x)-2)\frac{u_iu_j}{|Du|^2}\right)u_{ij}=0
\end{equation}
and showed that the viscosity solution is locally $C^{1,\alpha}$ regular by means of the equivalence between viscosity solutions to \eqref{1-2} and weak solutions to strong $p(x)$-Laplace equation
$$
\Delta^S_{p(x)}u=|Du|^{p(x)-2}\Delta^N_{p(x)}u.
$$
And the local $C^{1,\alpha}$ regularity of weak solutions of strong $p(x)$-Laplace equation has been obtained by Zhang-Zhou \cite{ZZ12}.  For more regularity results in elliptic situation, see for instance \cite{AR18,BD10,BD14,BM19,IS13,Wang94} and references therein.

On the other hand, regularity studies were extended to the parabolic normalized $p$-Laplace equation
\begin{equation}
\label{1-3}
\partial_tu=\Delta^N_pu.
\end{equation}
The existence, uniqueness as well as the long time behaviour of viscosity solutions and  the Lipschitz continuity in the spatial variables were investigated by Banerjee-Garofalo in \cite{BG13} and Does in \cite{Doe11}, respectively (see also \cite{BG15, BG2015, HL20, Juu11} for further results). More recently, Jin-Silvestre \cite{JS17} derived the local H\"{o}lder gradient estimates for Eq. \eqref{1-3} and  Jin-Silvestre-Imbert \cite{IJS19} extended the result to a more general equation
\begin{equation}
\label{1-4}
\partial_tu=|Du|^\gamma\left(\delta_{ij}+(p-2)\frac{u_i u_j}{|Du|^2}\right)u_{ij},
\end{equation}
where $p\in(1,+\infty)$ and $\gamma>-1$.  When $\gamma=0$, it is nothing but \eqref{1-3}; when $\gamma=p-2$, it is the usual parabolic $p$-Laplace equation
\begin{equation}
\label{1-5}
u_t={\rm div}(|Du|^{p-2}Du).
\end{equation}
It was well-known that viscosity solutions and weak solutions to \eqref{1-5} coincide (see \cite{JLM01}). Based on this equivalence and the $C^{1,\alpha}$ regularity of weak solutions to \eqref{1-5} in \cite{BF85,Wie86}, we find that the viscosity solutions are of class $C^{1,\alpha}$. For the inhomogeneous counterpart of \eqref{1-4}
\begin{equation}
\label{1-6}
\partial_tu-|Du|^\gamma\left(\delta_{ij}+(p-2)\frac{u_i u_j}{|Du|^2}\right)u_{ij}=f
\end{equation}
with $-1<\gamma<\infty$ and $1<p<\infty$, the local higher regularity properties of solutions to \eqref{1-6} have been studied in \cite{AP18,Att18,AR19}, provided that $f$ is bounded and continuous. For more results, one can refer to \cite{KKK14,LS15,MPR12,Ros11,Rud15} and references therein.

As interpreted in \cite{PR16,Hei18}, parabolic equations of the type considered in \eqref{1-1}  arise naturally from a two-player zero-sum stochastic differential game (SDG) with probabilities depending on space and time. It is defined in terms of an $n$-dimensional state process, and is driven by a $2n$-dimensional Brownian motion for $n\ge 2$.  As far as we know, the present setting is less studied and it exhibits interesting features both from the tug-of-war games and the mathematical viewpoint.
In particular, Parviainen-Ruosteenoja \cite{PR16} proved  the H\"{o}lder and Harnack estimates for a more general game that was called $p(x,t)$-game without using the PDE techniques and showed that the value functions of the game converge to the unique viscosity solution of the Dirichlet problem to the normalized $p(x, t)$-parabolic equation
$$
(n+p(x,t))u_t(x,t)=\Delta^N_{p(x,t)}u(x,t).
$$
In addition, Heino \cite{Hei18} formulated a stochastic differential game in continuous time and obtained that the viscosity solution to a terminal value problem involving the parabolic normalized $p(x,t)$-Laplace operator is unique under suitable assumptions. However, whether or not the spatial gradient $\nabla u$ of \eqref{1-1} is H\"{o}lder continuous was still unknown. In this paper we answer this question  and prove the interior H\"{o}lder continuity for the spatial gradient of viscosity solutions to \eqref{1-1}.


Let $Q_r:=B_r\times(-r^2,0]\subset\mathbb{R}^{n+1}$ be a parabolic cylinder, where $B_r$ is a ball in $\mathbb{R}^n$ centered at the origin with the radius $r>0$.
Our main result is stated as follows.
\begin{theorem}
\label{main}
Assume that $u$ is a viscosity solution to \eqref{1-1} in $Q_1$. If $1<p_-\leq p_+<\infty$ and $p(x,t)\in C^1(\overline{Q_1})$, then there exist two constants $\alpha\in (0,1)$ and $C$, both only depending on $n,p_-,p_+$, such that
$$
\|Du\|_{C^\alpha(Q_{1/2})}\leq C\|u\|_{L^\infty(Q_1)}
$$
and
$$
\sup_{Q_{1/2}}\frac{|u(x,t)-u(x,s)|}{|t-s|^{\frac{1+\alpha}{2}}}\leq C\|u\|_{L^\infty(Q_1)}.
$$
\end{theorem}

We would like to mention that our proof is much influenced by the ideas developed in \cite{JS17}. To avoid the problem of vanishing gradient, we first approximate \eqref{1-1}  with a regularized problem \eqref{3-1} below. Then we try to derive uniform {\em a priori} estimates regarding \eqref{3-1}, so that we could pass to the limit through compactness argument eventually. Specifically, we verify that the oscillation of the spatial gradient decreases in a sequence of the shrinking parabolic cylinders. The iteration process is divided into two scenarios: either the gradient $Du$ is close to a fixed vector $e$ in a large portion of $Q_{\tau^k}$, or it does not. We then have to combine these two alternatives to get the final result. In fact, by virtue of the similar structure between \eqref{1-1} and \eqref{3-1}, we focus mainly on showing the improvement of oscillation for $|Du|$ (Lemmas \ref{lem3-1} and \ref{lem3-2}) and demonstrate the higher H\"{o}lder regularity of solutions to the original equation \eqref{1-1} via approximations. It is worth pointing out that the comparison principle and stability of viscosity solutions play an important role in the proof of Theorem \ref{thm3-6}.
To the best of our knowledge, the proof of comparison principle of \eqref{1-1} is new, which is also independent interest.


This paper is organized as follows. In Section \ref{sec-2}, we give the definition of viscosity solutions to \eqref{1-1} and state some known results that will be used later. Section \ref{sec-3} is devoted to show the H\"{o}lder gradient regularity of \eqref{1-1} under the assumption that $\|Dp\|_{L^\infty(Q_1)}$ is small first, then consummating the conclusion for all $p(x,t)\in C^1(\overline{Q_1})$. In Section \ref{sec-4}, we prove the comparison principle and stability of viscosity solutions to \eqref{1-1}, which are the indispensable ingredients for  the proof of Theorem \ref{thm3-6}.

\section{Preliminaries}
\label{sec-2}

Because Eq. \eqref{1-1} is not in divergence form, the concept of weak solutions with test functions under the integral sign is problematic. Thus,  in this section we first recall the definition of viscosity solution to \eqref{1-1}.

\begin{definition} [viscosity solution]
A lower (resp. upper) semicontinuous function u in $Q_1$ is a viscosity supersolution (resp. subsolution) to \eqref{1-1}, if for any $\varphi\in C^2(Q_1)$, $u-\varphi$ reaches the local minimum at $(x_0,t_0)\in Q_1$, then when $D\varphi(x_0,t_0)\neq0$, it holds that
$$
\varphi_t\geq(\leq,resp.)\Delta \varphi+(p(x,t)-2)\left\langle D^2\varphi\frac{D\varphi}{|D\varphi|},\frac{D\varphi}{|D\varphi|}\right\rangle
$$
at $(x_0,t_0)$; when $D\varphi(x_0,t_0)=0$, it holds that
$$
\varphi_t\geq(\leq,resp.)\Delta \varphi+(p(x,t)-2)\langle D^2\varphi q,q\rangle
$$
at $(x_0,t_0)$ for some $q\in \overline{B_1}(0)\subset\mathbb{R}^n$. A function $u$ is a viscosity solution to \eqref{1-1} if and only if it is both viscosity supper- and subsolution.
\end{definition}

Next, we state some known results about solutions of linear uniformly parabolic equations, which will be used later.
Consider the equation
\begin{equation}
\label{2-1}
u_t-a_{ij}(x,t)u_{ij}=0 \quad \text{in }  Q_1,
\end{equation}
where the coefficient $a_{ij}$ is uniformly parabolic, i.e., there exist two constants $0<\lambda\leq\Lambda<\infty$ such that
\begin{equation}
\label{2-2}
\lambda I\leq a_{ij}(x,t)\leq\Lambda I \quad \text{for all }  (x,t)\in Q_1.
\end{equation}

We begin with the following two lemmas (see \cite{JS17}).
\begin{lemma}
\label{lem2-1}
Let $u\in C(\overline{Q_1})$ be a solution to \eqref{2-1} satisfying \eqref{2-2} and $A$ be a positive constant. If
$$
{\rm osc}_{B_1}u(\cdot,t)\leq A
$$
for any $t\in[-1,0]$, then we have
$$
{\rm osc}_{Q_1}u(x,t)\leq CA,
$$
where $C>0$ depends only on $n,\Lambda$.
\end{lemma}

\begin{lemma}
\label{lem2-2}
Let $\eta,u$ be a positive constant and a smooth solution to \eqref{2-1} satisfying \eqref{2-2} respectively. Suppose $|Du|\leq 1$ in $Q_1$ and
$$
|\{(x,t)\in Q_1:|Du-e|>\varepsilon_0\}|\leq\varepsilon_1
$$
for some $e\in\mathbb{S}^{n-1}$ and two positive constants $\varepsilon_0,\varepsilon_1$. Then there is a constant $a\in\mathbb{R}$ such that
$$
|u(x,t)-a-e\cdot x|\leq \eta
$$
for any $(x,t)\in Q_{1/2}$, provided that $\varepsilon_0,\varepsilon_1$ are small enough. Here $\varepsilon_0,\varepsilon_1$ depend on $n,\lambda,\Lambda$ and $\eta$.
\end{lemma}

Subsequently, we present an important conclusion about improvement of oscillation for solution to \eqref{2-1}.

\begin{lemma} [\label{lem2-3}\cite{JS17}]
Assume $u\in C(Q_1)$ is a nonnegative supersolution to \eqref{2-1} satisfying \eqref{2-2}. For any $0<\mu<1$, there is $\tau\in (0,1)$ depending only on $n,\mu$ and $\gamma>0$ depending on $n,\mu,\lambda,\Lambda$ such that if
$$
|\{(x,t)\in Q_1:u\geq1\}|>\mu|Q_1|,
$$
then it holds that
$$
u\geq\gamma \quad \text{in } Q_\tau.
$$
\end{lemma}

We end this section by the following boundary estimates of solutions to \eqref{2-1} utilized in the proof of Theorem \ref{thm3-6}.

\begin{lemma}   [\label{lem2-4}\cite{JS17}]
Suppose that $u\in C(\overline{Q_1})$ is a solution to \eqref{2-1} satisfying \eqref{2-2} and that $\rho$ is a modulus of continuity of boundary value $\varphi:=u\mid_{\partial_p Q_1}$. Then there is another modulus of continuity $\rho^*$ that depends on $n,\lambda,\Lambda,\rho,\|\varphi\|_{L^\infty(\partial_p Q_1)}$ such that
$$
|u(x,t)-u(y,s)|\leq \rho^*(|x-y|\vee \sqrt{|t-s|})
$$
for any $(x,t),(y,s)\in \overline{Q_1}$. Here $a\vee b$ denotes $\max\{a,b\}$.
\end{lemma}

\section{H\"{o}lder regularity of spatial gradients}
\label{sec-3}

To avoid the lack of smoothness of viscosity solutions to \eqref{1-1}, we first regularize  the Eq. \eqref{1-1} to
\begin{equation}
\label{3-1}
u_t=\left(\delta_{ij}+(p(x,t)-2)\frac{u_i u_j}{|Du|^2+\varepsilon^2}\right)u_{ij}
\end{equation}
with $\varepsilon>0$. For later convenience, we denote
\begin{equation*}
a^\varepsilon_{ij}:=a^\varepsilon_{ij}(x,t,Du)=\delta_{ij}+(p(x,t)-2)\frac{u_i u_j}{|Du|^2+\varepsilon^2}
\end{equation*}
with $u_i$ being the $i$-th component of $Du$.

Now we present the interior Lipschtiz regularity of solutions to \eqref{3-1}.

\begin{lemma}
\label{lem3-1}
Let $u$ be a smooth solution to \eqref{3-1} in $Q_4$ with $\varepsilon\in(0,1)$. Then there is a constant $C>0$, which depends on $n,p_-,p_+$ and $\|u\|_{L^\infty(Q_4)}$, such that
$$
|u(x,t)-u(y,t)|\leq C|x-y|
$$
for each $(x,t),(y,t)\in Q_3$ and $|x-y|<1$.
\end{lemma}
\begin{proof}
As the proof of Lipschitz estimates in Section 2 in \cite{IJS19}, this conclusion holds as well. It is enough to notice that the matrix
$$
I+(p(x,t)-2)\frac{q\otimes q}{|q|^2+\varepsilon^2}  \quad (q\in \mathbb{R}^n)
$$
is uniformly elliptic.
\end{proof}

\begin{remark}
It follows from Lemma \ref{lem3-1} that the spatial gradient $Du$ is bounded. By normalization we may assume $|Du|\leq1$ below.
\end{remark}

In what follows, we first show that the solutions of Eq. \eqref{1-1} are of $C^{1,\alpha}$ for the case that $\|Dp\|_{L^\infty(Q_1)}$ is small enough. Then via doing a scaling work, we verify the solutions are $C^{1,\alpha}$-regular when $p(x,t)\in C^1(\overline{Q_1})$, i.e., $Dp$ exhibits a general boundness.

\subsection{H\"{o}lder regularity of spatial gradient for the case that $\|Dp\|_{L^\infty(Q_1)}$ is small enough}

We now derive the improvement of oscillation of $Du\cdot e$.

\begin{lemma}
\label{lem3-2}
Suppose $u$ is a smooth viscosity solution to \eqref{3-1} in $Q_1$. For every $0<l<1,\mu>0$, if $p(x,t)\in C^1(\overline{Q_1})$ and $\|Dp\|_{L^\infty(Q_1)}\leq \beta$, where $\beta$ is a small enough constant depending on $n,p_-,p_+,\mu$ and $l$, then we can conclude that there are two positive constants $\tau$ and $\delta$, the former depending only on $n,\mu$ and the latter depending on $n,p_-,p_+,\mu$ and $l$, such that for arbitrary $e\in\mathbb{S}^{n-1}$, if
$$
|\{(x,t)\in Q_1: Du\cdot e\leq l\}|>\mu|Q_1|,
$$
we have
$$
Du\cdot e<1-\delta \quad \text{in }  Q_\tau.
$$
\end{lemma}
\begin{proof}
Let
$$
a^\varepsilon_{ij,m}:=\frac{\partial a^\varepsilon_{ij}(x,t,Du)}{\partial u_m}
=(p(x,t)-2)\left(\frac{\delta_{im}u_j+\delta_{jm}u_i}{|Du|^2+\varepsilon^2}-\frac{2u_iu_ju_m}{(|Du|^2+\varepsilon^2)^2}\right).
$$
Differentiating Eq. \eqref{3-1} in $x_k$ derives
$$
(u_k)_t=a^\varepsilon_{ij}(u_k)_{ij}+a^\varepsilon_{ij,m}u_{ij}(u_k)_m+p_k\frac{u_iu_j}{|Du|^2+\varepsilon^2}u_{ij},
$$
where $p_k:=\frac{\partial p(x,t)}{\partial x_k}$.
Define
$$
w=(Du\cdot e-l+\rho|Du|^2)^+
$$
with $\rho=\frac{l}{4}$. Then for the function $Du\cdot e-l$ we have
$$
(Du\cdot e-l)_t=a^\varepsilon_{ij}(Du\cdot e-l)_{ij}
+a^\varepsilon_{ij,m}u_{ij}(Du\cdot e-l)_m+Dp\cdot e\frac{u_iu_ju_{ij}}{|Du|^2+\varepsilon^2},
$$
and for $|Du|^2$ derive
$$
(|Du|^2)_t=a^\varepsilon_{ij}(|Du|^2)_{ij}+a^\varepsilon_{ij,m}u_{ij}(|Du|^2)_m
+Dp\cdot Du\frac{u_iu_ju_{ij}}{|Du|^2+\varepsilon^2}-2a^\varepsilon_{ij}u_{ki}u_{kj},
$$
where $Dp$ denotes the spatial gradient of $p(x,t)$.

Merging the previous two identities arrives at in the region $\Omega_+:=\{(x,t)\in Q_1:w>0\}$
$$
w_t=a^\varepsilon_{ij}w_{ij}+a^\varepsilon_{ij,m}u_{ij}w_m
+Dp\cdot(e+\rho Du)\frac{u_iu_ju_{ij}}{|Du|^2+\varepsilon^2}-2\rho a^\varepsilon_{ij}u_{ki}u_{kj}.
$$
Noting that $|Du|>\frac{l}{2}$ in $\Omega_+$, we have
\begin{equation}
\label{3-2}
|a^\varepsilon_{ij,m}|\leq\frac{4|p(x,t)-2|}{l}\leq\frac{4}{l}\max\{|p_+-2|,|p_--2|\}=:\frac{4}{l}b.
\end{equation}
By Cauchy-Schwarz inequality and \eqref{3-2}, we obtain
\begin{align*}
w_t&\leq a^\varepsilon_{ij}w_{ij}+\frac{4}{l}b|Dw|\sum^n_{i,j}|u_{ij}|+(1+\rho)|Dp|\frac{|\langle D^2u\cdot Du,Du\rangle|}{|Du|^2+\varepsilon^2}-2\rho a^\varepsilon_{ij}u_{ki}u_{kj}\\
&\leq a^\varepsilon_{ij}w_{ij}+\varepsilon|D^2u|^2+\frac{4n^2b^2}{\varepsilon l^2}|Dw|^2+(1+\rho)|Dp||D^2u|\\
&\quad-2\rho\left(|D^2u|^2+(p(x,t)-2)\frac{|D^2u\cdot Du|^2}{|Du|^2+\varepsilon^2}\right)\\
&\leq a^\varepsilon_{ij}w_{ij}+2\varepsilon|D^2u|^2+\frac{4n^2b^2}{\varepsilon l^2}|Dw|^2+\frac{(1+\rho)^2}{4\varepsilon}|Dp|^2\\
&\quad-2\rho\left(|D^2u|^2+(p(x,t)-2)\frac{|D^2u\cdot Du|^2}{|Du|^2+\varepsilon^2}\right).
\end{align*}

Denote
$$
\Omega_1:=\{p(x,t)\geq2\}\cap\Omega_+ \quad \text{and} \quad \Omega_2:=\{p(x,t)<2\}\cap\Omega_+.
$$

In $\Omega_1$, we get
\begin{equation*}
w_t\leq a^\varepsilon_{ij}w_{ij}+2\varepsilon|D^2u|^2+\frac{4n^2b^2}{\varepsilon l^2}|Dw|^2+\frac{(1+\rho)^2}{4\varepsilon}|Dp|^2-2\rho|D^2u|^2;
\end{equation*}

In $\Omega_2$, we have
\begin{align*}
w_t&\leq a^\varepsilon_{ij}w_{ij}+2\varepsilon|D^2u|^2+\frac{4n^2b^2}{\varepsilon l^2}|Dw|^2+\frac{(1+\rho)^2}{4\varepsilon}|Dp|^2+2\rho(1-p(x,t))|D^2u|^2\\
&\leq a^\varepsilon_{ij}w_{ij}+2\varepsilon|D^2u|^2+\frac{4n^2b^2}{\varepsilon l^2}|Dw|^2+\frac{(1+\rho)^2}{4\varepsilon}|Dp|^2+2\rho(1-p_-)|D^2u|^2.
\end{align*}

\textbf{Case 1.} If $2\leq p_-$, then we obtain by choosing $\varepsilon=\rho$
\begin{align*}
w_t&\leq a^\varepsilon_{ij}w_{ij}+\frac{4n^2b^2}{\rho l^2}|Dw|^2+\frac{(1+\rho)^2}{4\rho}|Dp|^2\\
&\leq a^\varepsilon_{ij}w_{ij}+\frac{4n^2b^2}{\rho l^2}|Dw|^2+\frac{(1+\rho)^2}{4\rho}M^2
\end{align*}
in $\Omega_+$, where $b=p_+-2$ and $M=\|Dp\|_{L^\infty(Q_1)}$. Let
$$
\overline{c}=\frac{(1+\rho)^2}{4\rho}M^2.
$$
Thereby it satisfies in the viscosity sense that
\begin{equation*}
w_t\leq a^\varepsilon_{ij}w_{ij}+\frac{4n^2b^2}{\rho l^2}|Dw|^2+\overline{c} \quad \text{in }  Q_1.
\end{equation*}
Set $\overline{w}=1-l+\rho+\overline{c}$ and $\nu=\frac{c_1}{\rho l^2}$, where $c_1$ will be chosen later. Define
$$
U=\frac{1}{\nu}(1-e^{\nu(w-\overline{c}t-\overline{w})}).
$$

Observe that
$$
a^\varepsilon_{ij}w_{ij}+\nu a^\varepsilon_{ij}w_iw_j\geq a^\varepsilon_{ij}w_{ij}+\nu|Dw|^2.
$$
Hence we can take $c_1=4n^2(p_+-2)^2$ such that
$$
U_t\geq a^\varepsilon_{ij}U_{ij} \quad \text{in }  Q_1
$$
in the viscosity sense. Obviously, $U\geq0$ in $Q_1$.

If $Du\cdot e\leq l$, then it follows that
$$
|\{(x,t)\in Q_1:U\geq\nu^{-1}(1-e^{\nu(l-1)})\}|>\mu|Q_1|.
$$
Thus we can conclude from Lemma \ref{lem2-3} that there exist two constants $\tau,\gamma_0>0$ such that
$$
U\geq\nu^{-1}(1-e^{\nu(l-1)})\gamma_0 \quad \text{in }  Q_{\tau},
$$
where $\tau$ and $\gamma_0$ depend on $\mu,n$ and $n,p_-,p_+,\mu$ separately. Since $w\leq \overline{w}+\overline{c}t$, we derive
$$
U\leq\overline{w}-w+\overline{c}t.
$$
Therefore in $Q_\tau$ we get
$$
Du\cdot e+\rho|Du|^2\leq1+\rho-\nu^{-1}(1-e^{\nu(l-1)})\gamma_0+\overline{c}+\overline{c}t.
$$
By $|Du\cdot e|\leq|Du|$, the above inequality becomes
$$
Du\cdot e+\rho Du\cdot e\leq1+\rho-\nu^{-1}(1-e^{\nu(l-1)})\gamma_0+\overline{c}\quad\text{in } Q_\tau.
$$
Furthermore,
$$
Du\cdot e\leq\frac{-1+\sqrt{1+4\rho(1+\rho-\nu^{-1}(1-e^{\nu(l-1)})\gamma_0+\overline{c})}}{2\rho} \quad\text{in }  Q_\tau.
$$
Here we need
$$
\frac{-1+\sqrt{1+4\rho(1+\rho-\nu^{-1}(1-e^{\nu(l-1)})\gamma_0+\overline{c})}}{2\rho}<1.
$$
Namely,
\begin{eqnarray*}
&&\qquad \overline{c}<\nu^{-1}(1-e^{\nu(l-1)})\gamma_0\\
&&\Longleftrightarrow \frac{(1+\rho)^2}{4\rho}M^2<\nu^{-1}(1-e^{\nu(l-1)})\gamma_0\\
&&\Longleftrightarrow M^2<\frac{4\rho\gamma_0}{\nu(1+\rho)^2}(1-e^{\nu(l-1)}),
\end{eqnarray*}
where $\nu=\frac{4n^2}{\rho l^2}(p_+-2)^2$. In other words, when $M:=\|Dp\|_{L^\infty(Q_1)}$ is small enough depending on $n,p_-,p_+,l,$ and $\mu$, we get
$$
Du\cdot e\leq 1-\delta \quad \text{in }   Q_\tau,
$$
where $\delta>0$ depends on $n,p_-,p_+,l,$ and $\mu$.

\textbf{Case 2.} If $1<p_-<2$, we obtain
$$
w_t\leq a^\varepsilon_{ij}w_{ij}+\frac{4n^2b^2}{\rho l^2(p_--1)}|Dw|^2+\frac{(1+\rho)^2}{4\rho(p_--1)}|Dp|^2 \quad \text{in }  \Omega_+,
$$
where $b=\max\{|p_+-2|,|p_--2|\}$. Let
$$
\widehat{c}=\frac{(1+\rho)^2}{4\rho(p_--1)}M^2.
$$
It follows that
$$
w_t\leq a^\varepsilon_{ij}w_{ij}+\frac{4b^2}{\rho l^2(p_--1)}|Dw|^2+\widehat{c} \quad \text{in }  Q_1
$$
in the viscosity sense.

Notice that
$$
a^\varepsilon_{ij}w_{ij}+\nu a^\varepsilon_{ij}w_iw_j\geq a^\varepsilon_{ij}w_{ij}+\nu(p_--1)|Dw|^2
$$
with $\nu=\frac{c_2}{\rho l^2(p_--1)}>0$, where $c_2$ is a constant determined later. Denote $\widehat{w}=1-l+\rho+\widehat{c}$ and $V=\frac{1}{\nu}(1-e^{\nu(w-\widehat{c}t-\widehat{w})})$. We take $c_2=\frac{4n^2b^2}{p_--1}$ such that
$$
V_t\geq a^\varepsilon_{ij}V_{ij} \quad \text{in }  Q_1
$$
in the viscosity sense. Apparently, $V\geq0$ in $Q_1$.

For $Du\cdot e\leq l$, by the assumption we have
$$
|\{(x,t)\in Q_1:V\geq\nu^{-1}(1-e^{\nu(l-1)})\}|>\mu|Q_1|.
$$
Using again Lemma \ref{lem2-3} deduces that there are two positive constants $\tau$ and $\gamma_0$, depending respectively on $\mu,n$ and $n,p_-,p_+,\mu,l$, such that
$$
V\geq \nu^{-1}(1-e^{\nu(l-1)})\gamma_0 \quad \text{in }  Q_\tau.
$$
We further obtain
$$
Du\cdot e+\rho(Du\cdot e)^2\leq1+\rho-\nu^{-1}(1-e^{\nu(l-1)})\gamma_0+\widehat{c} \quad \text{in }  Q_\tau.
$$
Thus
$$
Du\cdot e\leq\frac{-1+\sqrt{1+4\rho(1+\rho-\nu^{-1}(1-e^{\nu(l-1)})\gamma_0+\widehat{c})}}{2\rho} \quad\text{in }  Q_\tau.
$$
Analogous to Case 1, for $M=\|Dp\|_{L^\infty(Q_1)}$ sufficiently small and depending on $n,p_-,p_+,l$ and $\mu$, we arrive at
$$
Du\cdot e\leq 1-\delta \quad \text{in }  Q_\tau,
$$
where $\delta>0$ depends on $n,p_-,p_+,l,$ and $\mu$. We now complete the proof.
\end{proof}

\begin{remark}
In the case that $p_-\geq2$, we note that
$$
I\leq(p_--1)I\leq(a^\varepsilon_{ij}(x,t,q))_{n\times n}\leq(p_+-1)I
$$
for all $\varepsilon\in(0,1),q\in\mathbb{R}^n$ and $(x,t)\in Q_1$, so the constant $\gamma_0$ appearing in Case 1 may not depend on $p_-$.
\end{remark}

\begin{lemma}
\label{lem3-3}
Let $u$ be a smooth solution of \eqref{3-1} in $Q_1$. For any $0<l<1,\mu>0$, when $\|Dp\|_{L^\infty(Q_1)}\leq \beta$ with $\beta$ being a sufficiently small constant depending on $n,p_-,p_+,l,\mu$, there is $\tau>0$ (small) depending on $n,\mu$, and $\delta>0$ depending on $n,p_-,p_+,l,\mu$, such that for any nonnegative integer $k$, if
$$
|\{(x,t)\in Q_{\tau^i}: Du\cdot e\leq l(1-\delta)^i\}|>\mu|Q_{\tau^i}| \quad \text{for all }  e\in\mathbb{S}^{n-1},
$$
and $i=0,1,\cdots,k$, then
$$
|Du|<(1-\delta)^{i+1} \quad \text{in }  Q_{\tau^{i+1}}
$$
for all $i=0,1,\cdots,k$.
\end{lemma}
\begin{proof}
We prove this lemma by induction. For $k=0$, the conclusion holds obviously by Lemma \ref{lem3-2}. Suppose the conclusion is true for $i=0,1,\cdots,k-1$. We are going to verify it for $i=k$. Set
$$
v(x,t):=\frac{1}{\tau^k(1-\delta)^k}u(\tau^kx,\tau^{2k}t).
$$
Then $v$ satisfies
$$
v_t=\Delta v+(h_k(x,t)-2)\frac{v_iv_j}{|Dv|^2+\varepsilon^2(1-\delta)^{-2k}}v_{ij} \quad \text{in }  Q_1,
$$
where $h_k(x,t)=p(\tau^kx,\tau^{2k}t)$. We can see from the induction assumptions that $|Dv|<1$ in $Q_1$, and
$$
|\{(x,t)\in Q_1: Dv\cdot e\leq l\}|>\mu|Q_1| \quad \text{for all }  e\in\mathbb{S}^{n-1}.
$$
Furthermore, we observe
\begin{equation*}
1<p_-\leq h_k(x,t)\leq p_+<\infty
\end{equation*}
and
\begin{equation*}
|Dh_k(x,t)|=|\tau^kDp(y,s)|\leq\tau^k\|Dp\|_{L^\infty(Q_1)},
\end{equation*}
where $(y,s)=(\tau^kx,\tau^{2k}t)$ and $(x,t)\in Q_1$. Hence from Lemma \ref{lem3-2} we get
$$
Dv\cdot e\leq1-\delta  \quad \text{in }  Q_\tau
$$
for all $e\in\mathbb{S}^{n-1}$.
Namely, $|Dv|\leq1-\delta$ in $Q_\tau$. Rescaling back, we arrive at
$$
|Du|<(1-\delta)^{k+1} \quad \text{in } Q_{\tau^{k+1}}.
$$
We finish the proof.
\end{proof}

\begin{remark}
Noting that $0<\tau<1$, when $Dp(x,t)$ is bounded, we can see that
$$
|Dh_k(x,t)|\rightarrow0 \quad \text{uniformly in }  Q_1,
$$
by sending $k\rightarrow\infty$. That is to say, for $k$ large enough, we could remove the restriction that $\|Dp\|_{L^\infty(Q_1)}$ is sufficiently small.
\end{remark}

We shall present a lemma, playing an important role in the proof of Theorem \ref{thm3-5}, which is a regularity estimate of small perturbation solutions of fully nonlinear parabolic equations.
\begin{lemma}
\label{lem3-4}
Let $u$ be a smooth solution to \eqref{3-1} in $Q_1$. For $\gamma=\frac{1}{2}$, there are two positive constants $\eta$ (small) and $C$ (large), both depending on $n,p_-,p_+$ and $\|D_{x,t}p\|_{L^\infty(Q_1)}$ such that if a linear function $L(x)$ with $\frac{1}{2}\leq|DL|\leq2$ satisfies
$$
\|u(x,t)-L(x)\|_{L^\infty(Q_1)}\leq\eta,
$$
then
$$
\|u(x,t)-L(x)\|_{C^{2,1/2}(Q_{1/2})}\leq C.
$$
\end{lemma}
\begin{proof}
We can reach this conclusion from Corollary 1.2 in \cite{Wang13}, because $L(x)$ is a solution to \eqref{3-1} as well.
\end{proof}

\begin{remark}
From the Lemmas \ref{lem3-1} and \ref{lem3-2} above, we find that $\|Dp\|_{L^\infty(Q_1)}$ is small enough, so in Lemma \ref{lem3-4} we may assume that $\|D_{x,t}p\|_{L^\infty(Q_1)}$ is smaller than some sufficiently large constant determined so that we can substitute $\|D_{x,t}p\|_{L^\infty(Q_1)}$ by that constant.
\end{remark}

In the following, we give a uniformly {\em a priori} estimate for the solution to Eq. \eqref{3-1}.
\begin{theorem}
\label{thm3-5}
Let $u$ be a smooth solution to \eqref{3-1} in $Q_1$. Suppose that $p(x,t)\in C^1(\overline{Q_1})$ and $\|Dp\|_{L^\infty(Q_1)}\leq \beta$, where $\beta$ is a sufficiently small constant depending only on $n,p_-,p_+$. Then there are two positive constants $\alpha,C$, both of which depend on $n,p_-,p_+$, such that
$$
\|Du\|_{C^\alpha(Q_{1/2})}\leq C(\|u\|_{L^\infty(Q_1)}+\varepsilon)
$$
and
$$
\sup_{Q_{1/2}}\frac{|u(x,t)-u(x,s)|}{|t-s|^{\frac{1+\alpha}{2}}}\leq C(\|u\|_{L^\infty(Q_1)}+\varepsilon).
$$
\end{theorem}
\begin{proof}
As the proof of Theorem 4.5 in \cite{JS17}, we can first deduce $Du\in C^\alpha(Q_{1/2})$ by combining Lemma \ref{lem3-3} and Lemmas \ref{lem2-2}, \ref{lem3-4}. To this end, we choose $\eta$ as the one in Lemma \ref{lem3-4} with $\|D_{x,t}p\|_{L^\infty(Q_1)}$ replaced by some large constant fixed. And then we take $\varepsilon_0,\varepsilon_1>0$ such small constants that Lemma \ref{lem2-2} holds. Next, we determine the constants $l$ and $\mu$ to be $1-{\varepsilon^2_0}/2$ and ${\varepsilon_1}/{|Q_1|}$ respectively.

Terminally, by $Du\in C^\alpha(Q_{1/2})$ and using Lemma \ref{lem2-1}, we show that $u$ is $C^{\frac{1+\alpha}{2}}(Q_{1/2})$-regular in the $t$-variable.
\end{proof}

\begin{lemma}
\label{lem3-7}
Let $g\in C(\partial_pQ_1)$. For $\varepsilon>0$, there is a unique solution $u^\varepsilon\in C(\overline{Q_1})\cap C^\infty(Q_1)$ of Eq. \eqref{3-1} satisfying $u^\varepsilon=g$ on $\partial_pQ_1$.
\end{lemma}

For this lemma, we observe that Eq. \eqref{3-1} is uniformly parabolic and the coefficients $a^\varepsilon_{ij}(x,t,Du)$ are smooth with bounded derivatives for every $\varepsilon>0$. So it can be concluded from the classical quasilinear equation theory (see Theorem 4.4 of \cite{LSU68}, page 560) and the Schauder estimates.

Combining the previous conclusions, we now could establish an important intermediate result as follows.
\begin{theorem}
\label{thm3-6}
Let $u$ be a viscosity solution of \eqref{1-1} in $Q_1$. Assume that $p(x,t)\in C^1(\overline{Q_1})$ and $\|Dp\|_{L^\infty(Q_1)}\leq\beta$ with $\beta$ being a small enough constant that depends on $n,p_-,p_+$. Then there are two positive constants $\alpha\in(0,1)$ and $C$, both depending on $n,p_-$ and $p_+$, such that
$$
\|Du\|_{C^\alpha(Q_{1/2})}\leq C\|u\|_{L^\infty(Q_1)}
$$
and
$$
\sup_{Q_{1/2}}\frac{|u(x,t)-u(x,s)|}{|t-s|^{\frac{1+\alpha}{2}}}\leq C\|u\|_{L^\infty(Q_1)}.
$$
\end{theorem}
\begin{proof}
Without loss of generality, we can suppose $u\in C(\overline{Q_1})$. It follows from Lemma \ref{lem3-7} that there is a unique viscosity solution $u^\varepsilon\in C(\overline{Q_1})\bigcap C^\infty(Q_1)$ to Eq. \eqref{3-1} such that $u^\varepsilon=u$ on $\partial_p Q_1$. Based on the proof of Theorem 1.1 in \cite{JS17}, we note that it suffices to show that $u^\varepsilon$ converges to $u$ uniformly in $\overline{Q_1}$ as $\varepsilon\rightarrow0$ (up to a subsequence). To this end, we shall make use of comparison principle and stability property for viscosity solution to \eqref{1-1}, which are two counterparts to Theorems 2.9 and 2.10 in \cite{JS17}. Fortunately, these two conclusions hold true, whose proof will be presented in Section \ref{sec-4}.
\end{proof}

\subsection{H\"{o}lder regularity of spatial gradient for the case that $Dp$ is bounded, i.e.,  $\|Dp\|_{L^\infty(Q_1)}\le M$}

Set
$$
\widetilde{u}(x,t):=u(\epsilon x, \epsilon^2 t), \quad \widetilde{p}(x,t):=p(\epsilon x, \epsilon^2 t)
$$
with $0<\epsilon<1$.  By a scaling argument for Eq. \eqref{1-1}, it follows that $\widetilde{u}$ satisfies (in the viscosity sense) that
\begin{equation}
\label{3-3}
\widetilde{u}_t=\left(\delta_{ij}+(\widetilde{p}(x,t)-2)\frac{\widetilde{u}_i\widetilde{u}_j}{|D\widetilde{u}|^2}\right)\widetilde{u}_{ij} \quad \text{in } Q_{\epsilon^{-1}}.
\end{equation}

 When $\|Dp\|_{L^\infty(Q_1)}\le M$ ($p\in C^{1}(\overline{Q_1})$ and $M$ is large), then
 $$
 \|D\widetilde{p}\|_{L^\infty(Q_{\epsilon^{-1}})}\leq \epsilon \|Dp\|_{L^\infty(Q_1)}\le \epsilon M<\beta
 $$
by choosing $\epsilon$ small enough. Observe that the structure of \eqref{3-3} is similar to that of \eqref{1-1}. This permits us to employ the previous results in subsection 3.1 to show the local $C^{1,\alpha}$-regularity of the solution $\widetilde{u}$ to \eqref{3-3}. Then by rescaling back, we can deduce that the solution $u$ to \eqref{1-1} is of $C^{1,\alpha}_{\rm loc}$ provided $\|Dp\|_{L^\infty(Q_1)}\le M$. Thereby we reach the  conclusion that if function $p(x,t)\in C^1(\overline{Q_1})$, then the viscosity solution to \eqref{1-1} is locally $C^{1,\alpha}$-regular.


As has been stated above, we now complete the proof of Theorem \ref{main}.
\begin{remark}
Corresponding to Lemmas \ref{lem3-2} and \ref{lem3-3}, we find that $\epsilon$ is a small constant depending not only on $n,p_-,p_+$ but also on $\mu,l$. However, this really does not matter, since from  the proof of Theorem \ref{thm3-5} we notice that the constants $\mu,l$ will be fixed. And then by virtue of a series of dependencies, $\epsilon$ will finally depend only on $n, p_-, p_+$.
\end{remark}



\section{Comparison principle and stability for viscosity solution}
\label{sec-4}
In this section, we shall prove the comparison principle and stability properties for viscosity solutions. We shall make use of Ishii-Lions' method to show the comparison principle.

Let  $\Omega$ be a bounded domain of $\mathbb R^n$. We denote a general parabolic cylinder by $\Omega_T:=\Omega\times[0,T)$, and $\partial_p\Omega_T$ denotes its parabolic boundary. 

\begin{theorem}[comparison principle]
\label{thm4-1}
Suppose the function $p(x,t)$ in Eq. \eqref{1-1} is Lipschitz continuous. Let $u$ be a viscosity subsolution and $v$ be a continuous viscosity supersolution to \eqref{1-1}. If $u\leq v$ on $\partial_p\Omega_T$, then we can conclude
\begin{equation}
\label{4-1}
u\leq v \quad\text{in } \Omega_T.
\end{equation}
\end{theorem}
\begin{proof}
For convenience, we can assume $v$ is a strict supersolution, i.e.,
$$
v_t-\left(\Delta v+(p(x,t)-2)\left\langle D^2v\frac{Dv}{|Dv|},\frac{Dv}{|Dv|}\right\rangle\right)>0
$$
in the viscosity sense by considering $w:=v+\frac{\varepsilon}{T-t}$ instead, and $w\rightarrow\infty$ as $t\rightarrow T$. Indeed, we suppose $\varphi\in C^2(\Omega_T)$  such that $w-\varphi$ has local minimum at $(x_0,t_0)\in\Omega_T$, then so does $v-\phi$ by letting $\phi(x,t):=\varphi(x,t)-\frac{\varepsilon}{T-t}$. Notice that
\begin{align*}
&D\phi(x_0,t_0)=D\varphi(x_0,t_0),\\
&D_t\phi(x_0,t_0)=D_t\varphi(x_0,t_0)-\frac{\varepsilon}{(T-t_0)^2},
\end{align*}
and
$$
D^2\phi(x_0,t_0)=D^2\varphi(x_0,t_0).
$$
Because of $v$ being a viscosity supersolution, we obtain
\begin{align*}
0&\leq D_t\phi(x_0,t_0)-\left(\mathrm{tr} D^2\phi(x_0,t_0)+(p(x_0,t_0)-2)\left\langle D^2\phi(x_0,t_0)\frac{\phi(x_0,t_0)}{|\phi(x_0,t_0)|},\frac{\phi(x_0,t_0)}{|\phi(x_0,t_0)|}\right\rangle\right)\\
&=D_t\varphi(x_0,t_0)-\frac{\varepsilon}{(T-t_0)^2}\\
&\quad-\left(\mathrm{tr} D^2\varphi(x_0,t_0)+(p(x_0,t_0)-2)\left\langle D^2\varphi(x_0,t_0)\frac{\varphi(x_0,t_0)}{|\varphi(x_0,t_0)|},\frac{\varphi(x_0,t_0)}{|\varphi(x_0,t_0)|}\right\rangle\right),
\end{align*}
when $D\varphi(x_0,t_0)\neq0$. Here we denote by $\mathrm{tr}M$ the trace of matrix $M$. Furthermore,
\begin{align*}
0&<\frac{\varepsilon}{(T-t_0)^2}\\
&\leq D_t\varphi(x_0,t_0)-\left(\mathrm{tr} D^2\varphi(x_0,t_0)+(p(x_0,t_0)-2)\left\langle D^2\varphi(x_0,t_0)\frac{\varphi(x_0,t_0)}{|\varphi(x_0,t_0)|},\frac{\varphi(x_0,t_0)}{|\varphi(x_0,t_0)|}\right\rangle\right).
\end{align*}
When $D\varphi(x_0,t_0)=0$, we get for $|\eta|\leq1\quad (\eta\in\overline{B_1(0)})$
$$
0\leq D_t\phi(x_0,t_0)-(\mathrm{tr} D^2\phi(x_0,t_0)+(p(x_0,t_0)-2)\langle D^2\phi(x_0,t_0)\cdot\eta,\eta\rangle).
$$
Namely,
$$
0<\frac{\varepsilon}{(T-t_0)^2}\leq D_t\varphi(x_0,t_0)-(\mathrm{tr} D^2\varphi(x_0,t_0)+(p(x_0,t_0)-2)\langle D^2\varphi(x_0,t_0)\cdot\eta,\eta\rangle).
$$
In conclusion, we have verified that $w:=v+\frac{\varepsilon}{T-t}$ is a strict supersolution.

To demonstrate this conclusion, we argue by contradiction. Suppose \eqref{4-1} is not valid. Then it holds that for some $(\widehat{x},\widehat{t})\in \Omega\times(0,T)$, we have
$$
\theta:=u(\widehat{x},\widehat{t})-v(\widehat{x},\widehat{t})=\sup_{\Omega_T}(u-v)>0.
$$
Set
$$
\Psi_j(x,y,t,s)=u(x,t)-v(y,s)-\Phi_j(x,y,t,s),
$$
where $\Phi_j(x,y,t,s)=\frac{j}{q}|x-y|^q+\frac{j}{2}(t-s)^2$ with $q>\max\{2,\frac{p_-}{p_--1}\}$.

Without loss of generality, in what follows, we take a special value of $q$, i.e.,  $q=4$. Let $(x_j,y_j,t_j,s_j)$ be the maximum point of $\Psi_j$ in $\overline{\Omega}\times\overline{\Omega}\times[0,T)\times[0,T)$. We can prove that $(x_j,y_j,t_j,s_j)\in\Omega\times\Omega\times(0,T)\times(0,T)$ and $(x_j,y_j,t_j,s_j)\rightarrow(\widehat{x},\widehat{x},\widehat{t},\widehat{t})$ as $j\rightarrow\infty$ by the Lemma 7.2 in \cite{CIL92}.

\textbf{Case 1.} If $x_j=y_j$, then
\begin{equation*}
\begin{split}
&0=D_x\Phi_j(x_j,y_j,t_j,s_j)=D_y\Phi_j(x_j,y_j,t_j,s_j),\\
&0=D^2_x\Phi_j(x_j,y_j,t_j,s_j)=D^2_y\Phi_j(x_j,y_j,t_j,s_j).
\end{split}
\end{equation*}
Observe that
$$
u(x_j,t_j)-v(y_j,s_j)-\Phi_j(x_j,y_j,t_j,s_j)\geq u(x_j,t_j)-v(y,s)-\Phi_j(x_j,y,t_j,s).
$$
Denote
$$
\Theta(y,s):=-\Phi_j(x_j,y,t_j,s)+\Phi_j(x_j,y_j,t_j,s_j)+v(y_j,s_j).
$$
Obviously, $v(y,s)-\Theta(y,s)$ reaches the local minimum at $(y_j,s_j)$. Due to $v$ a strict supersolution, we arrive at
\begin{align*}
0&<\partial_s\Theta(y_j,s_j)-(\mathrm{tr} D^2\Theta(y_j,s_j)+(p(y_j,s_j)-2)\langle D^2\Theta(y_j,s_j)\cdot\eta,\eta\rangle)\\
&=j(t_j-s_j)
\end{align*}
for $|\eta|\leq1$. Analogously, letting $\beta(x,t):=\Phi_j(x,y_j,t,s_j)-\Phi_j(x_j,y_j,t_j,s_j)+u(x_j,t_j)$, we can
obtain
$$
0\geq\partial_t\beta(x_j,t_j)=j(t_j-s_j).
$$
From the previous two inequalities, we get
$$
0<j(t_j-s_j)-j(t_j-s_j)=0.
$$
This is a contradiction.

\textbf{Case 2.} If $x_j\neq y_j$, we have the following results.

By Theorem of sums (see \cite{CIL92}), for every $\mu>0$, there are $X_j,Y_j\in\mathcal{S}^n$ such that
$$
(\partial_t\Phi_j,D_x\Phi_j,X_j)\in\overline{\mathcal{P}}^{2,+}u(x_j,t_j),\quad (-\partial_s\Phi_j,-D_y\Phi_j,Y_j)\in\overline{\mathcal{P}}^{2,-}v(y_j,s_j)
$$
and
\begin{equation*}
\left(\begin{array}{cc}
X_j & \\
 &-Y_j
\end{array}
\right)\leq D^2\phi_j+\frac{1}{\mu}(D^2\phi_j)^2,
\end{equation*}
where all the derivatives are computed at $(x_j,y_j,t_j,s_j)$ and
\begin{equation*}
D^2\phi_j=\left(\begin{array}{cc}
B&-B\\
-B&B
\end{array}
\right)
\end{equation*}
with $B:=j|x_j-y_j|^2I+2j(x_j-y_j)\otimes(x_j-y_j)$.
Further, taking $\mu=j$ gets
\begin{equation}
\label{4-2}
\begin{split}
\left(\begin{array}{cc}
X_j & \\
 &-Y_j
\end{array}
\right)\leq& j(|x_j-y_j|^2+2|x_j-y_j|^4)\left(\begin{array}{cc}
I&-I\\
-I&I
\end{array}
\right)\\
&+2j(1+8|x_j-y_j|^2)\left(\begin{array}{cc}
G&-G\\
-G&G
\end{array}
\right),
\end{split}
\end{equation}
where $G:=(x_j-y_j)\otimes(x_j-y_j)$.
Note that \eqref{4-2} implies for any $\xi,\zeta\in\mathbb{R}^n$
\begin{equation}
\label{4-3}
\langle X_j\xi,\xi\rangle-\langle Y_j\zeta,\zeta\rangle\leq(3j|x_j-y_j|^2+18j|x_j-y_j|^4)|\xi-\zeta|^2.
\end{equation}
By virtue of the equivalent definition of viscosity solution emphasized by terminology of semi jets, we obtain
\begin{equation}
\label{4-4}
-\partial_s\Phi_j-\left(\mathrm{tr} Y_j+(p(y_j,s_j)-2)\left\langle Y_j\frac{-D_y\Phi_j}{|D_y\Phi_j|},\frac{-D_y\Phi_j}{|D_y\Phi_j|}\right\rangle\right)>0,
\end{equation}
and
\begin{equation}
\label{4-5}
\partial_t\Phi_j-\left(\mathrm{tr} X_j+(p(x_j,t_j)-2)\left\langle X_j\frac{D_x\Phi_j}{|D_x\Phi_j|},\frac{D_x\Phi_j}{|D_x\Phi_j|}\right\rangle\right)\leq0.
\end{equation}
Here we observe that
$$
\partial_t\Phi_j=j(t_j-s_j)=-\partial_s\Phi_j
$$
and
$$
\eta_j:=D_x\Phi_j=-D_y\Phi_j=j|x_j-y_j|^2(x_j-y_j).
$$
$\eta_j$ is nonzero, which is crucial. Denote
$$
A(x,t,\eta):=I+(p(x,t)-2)\frac{\eta}{|\eta|}\otimes\frac{\eta}{|\eta|},
$$
which is positive definite so that it possesses matrix square root denoted by $A^\frac{1}{2}(x,t,\eta)$. We denote the $k$-th column of $A^\frac{1}{2}(x,t,\eta)$ as $A^\frac{1}{2}_k(x,t,\eta)$.
Subtracting \eqref{4-5} from \eqref{4-4}, we derive
\begin{equation}
\label{4-6}
\begin{split}
0&<\mathrm{tr}(A(x_j,t_j,\eta_j)X_j)-\mathrm{tr}(A(y_j,s_j,\eta_j)Y_j)\\
&=\sum^n_{k=1}X_jA^\frac{1}{2}_k(x_j,t_j,\eta_j)\cdot A^\frac{1}{2}_k(x_j,t_j,\eta_j)-\sum^n_{k=1}Y_jA^\frac{1}{2}_k(y_j,s_j,\eta_j)\cdot A^\frac{1}{2}_k(y_j,s_j,\eta_j)\\
&\leq Cj|x_j-y_j|^2\|A^\frac{1}{2}(x_j,t_j,\eta_j)-A^\frac{1}{2}(y_j,s_j,\eta_j)\|^2_2\\
&\leq \frac{Cj|x_j-y_j|^2}{(\lambda_{\rm min}(A^\frac{1}{2}(x_j,t_j,\eta_j))+\lambda_{\rm min}(A^\frac{1}{2}(y_j,s_j,\eta_j))^2}
\|A(x_j,t_j,\eta_j)-A(y_j,s_j,\eta_j)\|^2_2,
\end{split}
\end{equation}
where the penultimate inequality is obtained by \eqref{4-3} and the last inequality is derived from the local Lipschitz continuity of $A\mapsto A^\frac{1}{2}$ (see \cite{HJ85}, page 410). Here $\lambda_{\rm min}(M)$ denotes the smallest eigenvalue of a symmetric $n\times n$ matrix $M$.

Now we estimate
\begin{equation*}
\begin{split}
\|A(x_j,t_j,\eta_j)-A(y_j,s_j,\eta_j)\|^2_2&=\left\|(p(x_j,t_j)-p(y_j,s_j))\frac{\eta_j}{|\eta_j|}\otimes\frac{\eta_j}{|\eta_j|}\right\|^2_2\\
&=|(p(x_j,t_j)-p(y_j,s_j))|^2\\
&\leq C(|x_j-y_j|^2+|t_j-s_j|^2),
\end{split}
\end{equation*}
where in the last inequality we employ the condition that $p(x,t)$ is Lipschitz continuous, i.e., $|p(x,t)-p(y,s)|\leq C|(x-y,t-s)|$. Moreover,
\begin{equation*}
\lambda_{\rm min}(A^\frac{1}{2}(x,t,\eta))=(\lambda_{\rm min}(A(x,t,\eta))^\frac{1}{2}=\min\{1,\sqrt{p_--1}\}.
\end{equation*}
Hence \eqref{4-6} turns into
\begin{align*}
0&<\frac{Cj|x_j-y_j|^2}{4\min\{1,p_--1\}}(|x_j-y_j|^2+|t_j-s_j|^2)\\
&=Cj|x_j-y_j|^4+Cj|t_j-s_j|^2|x_j-y_j|^2.
\end{align*}

On the other hand, we note that
\begin{equation*}
\begin{split}
u(x_j,t_j)-v(x_j,t_j)&\leq \max_{\overline{\Omega}\times[0,T)}\{u(x,t)-v(x,t)\}\\
&\leq u(x_j,t_j)-v(y_j,s_j)-\frac{j}{4}|x_j-y_j|^4-\frac{j}{2}(t_j-s_j)^2.
\end{split}
\end{equation*}
So we further get
\begin{align*}
\frac{j}{4}|x_j-y_j|^4+\frac{j}{2}(t_j-s_j)^2&\leq v(x_j,t_j)-v(y_j,s_j)\\
&\rightarrow v(\widehat{x},\widehat{t})-v(\widehat{x},\widehat{t})=0,
\end{align*}
by sending $j\rightarrow\infty$, where we have assumed $v$ is continuous in $\Omega_T$.

Consequently, we reach a contradiction that
$$
0<Cj|x_j-y_j|^4+Cj|t_j-s_j|^2|x_j-y_j|^2\rightarrow0
$$
as $j\rightarrow\infty$, observing that both $x_j$ and $y_j$ converge to the point $\widehat{x}$.
\end{proof}

We now conclude this section with stability properties of viscosity solution.

\begin{theorem}[stability]
\label{thm4-2}
Let $\{u_i\}$ be a sequence of viscosity solutions to \eqref{3-1} in $Q_1$ with $\varepsilon_i\geq0$ that $\varepsilon_i\rightarrow0$, and $u_i\rightarrow u$ locally uniformly in $Q_1$. Then $u$ is a viscosity solution to \eqref{1-1} in $Q_1$.
\end{theorem}

\begin{proof}
We only show that $u$ is a viscosity supersolution of \eqref{1-1}. The proof of $u$ being a subsolution is similar to that. Suppose $\varphi\in C^2(Q_1)$ such that $u-\varphi$ attains a local minimum at $(x_0,t_0)\in Q_1$. We know, from $u_i$ converging to $u$ locally uniformly, that there is $(x_i,t_i)\rightarrow(x_0,t_0)$ such that
$$
u_i-\varphi \quad\text{has local minimum at }  (x_i,t_i).
$$

If $D\varphi(x_0,t_0)\neq0$, then by $u_i$ viscosity supersolution, we obtain
\begin{align*}
\partial_t\varphi(x_i,t_i)\geq& \mathrm{tr} D^2\varphi(x_i,t_i)+(p(x_i,t_i)-2)\\
&\cdot\left\langle D^2\varphi(x_i,t_i)\frac{D\varphi(x_i,t_i)}{(|D\varphi(x_i,t_i)|^2+i^{-2})^\frac{1}{2}},
\frac{D\varphi(x_i,t_i)}{(|D\varphi(x_i,t_i)|^2+i^{-2})^\frac{1}{2}}\right\rangle.
\end{align*}
Letting $i\rightarrow\infty$, the above inequality becomes
$$
\partial_t\varphi(x_0,t_0)\geq F(x_0,t_0,D\varphi(x_0,t_0),D^2\varphi(x_0,t_0)),
$$
where $F(x,t,\eta,X):=\mathrm{tr} X+(p(x,t)-2)\langle X\frac{\eta}{|\eta|},\frac{\eta}{|\eta|}\rangle$.

If $D\varphi(x_0,t_0)=0$, we divide the proof into two cases. When $D\varphi(x_i,t_i)\neq0$ for $i$ large enough, it follows that
\begin{align*}
\partial_t\varphi(x_i,t_i)\geq& \mathrm{tr} D^2\varphi(x_i,t_i)+(p(x_i,t_i)-2)\\
&\cdot\left\langle D^2\varphi(x_i,t_i)\frac{D\varphi(x_i,t_i)}{(|D\varphi(x_i,t_i)|^2+i^{-2})^\frac{1}{2}},
\frac{D\varphi(x_i,t_i)}{(|D\varphi(x_i,t_i)|^2+i^{-2})^\frac{1}{2}}\right\rangle.
\end{align*}
For some vector $\xi\in \mathbb{R}^n$ with $|\xi|\leq1$, we deduce by sending $i\rightarrow\infty$
$$
\partial_t\varphi(x_0,t_0)\geq \mathrm{tr}D^2\varphi(x_0,t_0)+(p(x_0,t_0)-2)\langle D^2\varphi(x_0,t_0)\xi,\xi\rangle.
$$
When $D\varphi(x_i,t_i)\equiv0$ for $i$ sufficiently large, by the definition of supersolution, we have
$$
\partial_t\varphi(x_i,t_i)\geq \mathrm{tr}D^2\varphi(x_i,t_i)+(p(x_i,t_i)-2)\langle D^2\varphi(x_i,t_i)\xi_i,\xi_i\rangle,
$$
where $\xi_i\in\mathbb{R}^n$ satisfies $|\xi_i|\leq1$. Thus it follows that for some vector $|\xi|\leq1$
$$
\partial_t\varphi(x_0,t_0)\geq \mathrm{tr}D^2\varphi(x_0,t_0)+(p(x_0,t_0)-2)\langle D^2\varphi(x_0,t_0)\xi,\xi\rangle,
$$
as $i\rightarrow\infty$. Therefore, we prove that $u$ is a viscosity supersolution.
\end{proof}

\section*{Acknowledgments}

The authors wish to thank Prof. Tianling Jin for some very helpful conversations on this work. This work was supported by the National Natural Science Foundation of China (No. 11671111).

\end{document}